\documentclass[
11pt,  reqno]{amsart}
\usepackage[margin=1.1in,marginparwidth=1.5cm, marginparsep=0.5cm]{geometry}

\usepackage{booktabs} 
\usepackage{microtype}
\usepackage{amssymb}
\usepackage{mathrsfs}

\usepackage{upgreek} 

\usepackage{color}
\usepackage[implicit=true]{hyperref}

\usepackage{xsavebox}

\usepackage{cases}

\allowdisplaybreaks[2]

\newcommand{\mathd}{\mathrm{d}}

\DeclareMathOperator{\GNS}{GNS}

\definecolor{gr}{rgb}   {0.,   0.69,   0.23 }
\definecolor{bl}{rgb}   {0.,   0.5,   1. }
\definecolor{mg}{rgb}   {0.85,  0.,    0.85}
\definecolor{yl}{rgb}   {0.8,  0.7,   0.}
\definecolor{or}{rgb}  {0.7,0.2,0.2}

\usepackage{tikz} 

\usepackage{marginnote}
\usepackage{scalerel} 

\usetikzlibrary{shapes.misc}
\usetikzlibrary{shapes.symbols}
\usetikzlibrary{decorations}
\usetikzlibrary{decorations.markings}

\newtheorem{theorem}{Theorem} [section]

\newtheorem{lemma}[theorem]{Lemma}
\newtheorem{proposition}[theorem]{Proposition}
\newtheorem{remark}[theorem]{Remark}

\newcommand{\noi}{\noindent}
\newcommand{\Z}{\mathbb{Z}}
\newcommand{\R}{\mathbb{R}}
\newcommand{\T}{\mathbb{T}}

\let\P= \undefined
\newcommand{\P}{\mathbf{P}}

\newcommand{\E}{\mathbb{E}}

\newcommand{\F}{\mathcal{F}}

\newcommand{\al}{\alpha}

\newcommand{\dl}{\delta}

\newcommand{\Dl}{\Delta}
\newcommand{\eps}{\varepsilon}

\newcommand{\s}{\sigma}

\newcommand{\ft}{\widehat}

\newcommand{\wt}{\widetilde}
\newcommand{\cj}{\overline}
\newcommand{\dx}{\partial_x}
\newcommand{\dt}{\partial_t}

\renewcommand{\l}{\ell}
\renewcommand{\o}{\omega}
\renewcommand{\O}{\Omega}

\newcommand{\les}{\lesssim}

\newcommand{\jb}[1]
{\langle #1 \rangle}

\newcommand{\ind}{\mathbf 1}

\newtheorem*{ackno}{Acknowledgements}

\numberwithin{equation}{section}
\numberwithin{theorem}{section}

\usepackage{comment}

\begin{document}
\baselineskip = 14pt

\title[Gibbs measure for the focusing fractional NLS]
{Gibbs measure for the focusing fractional NLS on the torus}

\author[R.~Liang and Y.~Wang]
{Rui Liang and Yuzhao Wang}

\address{
Rui Liang and Yuzhao Wang\\
School of Mathematics\\
Watson Building\\
University of Birmingham\\
Edgbaston\\
Birmingham\\
B15 2TT\\ United Kingdom}

\email{RXL833@student.bham.ac.uk; y.wang.14@bham.ac.uk}



\begin{abstract}
We study the construction  of the Gibbs measures for the {\it focusing} mass-critical fractional nonlinear Schr\"odinger equation on the multi-dimensional torus.
We identify the sharp mass threshold for normalizability and non-normalizability 
of the  focusing Gibbs measures,
which generalizes the influential works of Lebowitz-Rose-Speer (1988), Bourgain (1994), 
and Oh-Sosoe-Tolomeo (2021) on the one-dimensional nonlinear Schr\"odinger equations.
To this purpose, 
we establish an almost sharp fractional Gagliardo-Nirenberg-Sobolev inequality on the torus,
which is of independent interest. 

\end{abstract}

%
\maketitle

\vspace{-5mm}


\section{Introduction}
\label{SEC:1}

\subsection{Focusing Gibbs measures}
In this paper, 
we consider the focusing Gibbs measure $\rho_{s,p}$
on the $d$-dimensional torus $\T^d = (\R /  \Z)^d$,
formally given by
\begin{align}
\label{Gibbs}
d\rho_{s,p} (u) = Z_{s,p}^{-1} \exp \bigg(\frac1{p} \int_{\T^d} |u|^p dx \bigg) d\mu_s (u)
\end{align}

\noi
for $p > 2$,
where $Z_{s,p}$ is a normalization constant.  Here,  $\mu_s$ is the Gaussian probability measure 
with the density:
\begin{align}
\label{GPM}
d\mu_s = Z_s^{-1} e^{-\frac12 \int_{\T} |D^s u|^2 d x} d u 
= Z_s^{-1} \prod_{n\neq 0} e^{-\frac12 (2\pi |n|)^{2s} |\ft u(n)|^2} d \ft u (n),
\end{align}

\noi
where $D = \sqrt{-\Delta}$  and $\ft u (n)$ denotes the Fourier coefficient of $u$.
When $s=1$, the measure $\mu_s$ corresponds to the massless 
Gaussian free field on $\T^d$.
A typical function $u$ in the support of $\mu_s$
is given by the random Fourier series:
\begin{align}
 \label{random_s}
 u^\o(x) = \sum_{n \in \Z^d \setminus \{0\}} \frac{g_n (\o)}{(2\pi |n|)^s} e^{ 2\pi i n \cdot x},
\end{align}

\noi
where $\{g_n\}_{n\in \Z^d \backslash \{0\}}$
denotes a sequence of independent standard complex-valued Gaussian random variables on a probability space $(\O, \F, \P)$.
When $d = 1$, the expression \eqref{random_s} corresponds to 
the mean-zero Brownian loop for $s = 1$ (namely $u(0) = u(1)$)
and to the mean-zero fractional Brownian loop for $\frac 12 < s < \frac 32$.
See \cite[Section 5]{OSTz}.
A standard computation shows that $u$ in~\eqref{random_s}
belongs to $\dot W^{\s, p}(\T^d) \setminus \dot W^{s - \frac d2, p}(\T^d)$ for any 
$\s < s - \frac d2$ and 
$1\leq p \leq \infty$
almost surely, where 
$\dot W^{\s, p}(\T^d)$ denotes the homogeneous Sobolev space (= the Riesz potential space)
defined by the norm:
\[ \| u \|_{\dot W^{\s, p}(\T^d)} = \| D^{\s} u \|_{L^p(\T^d)}
= \big\| \F^{-1}((2\pi|n|)^{2\s} \ft u(n) )\big \|_{L^p(\T^d)}. \]

\noi
In particular,  when $s > \frac{d}2$, 
$u$ in the support of $\mu_s$ is almost surely a function, 
while 
it is merely a distribution when $s \leq \frac d2$.
In the latter case, 
 the potential 
energy $\frac1{p} \int_{\T^d} |u|^p dx $ 
in \eqref{Gibbs} does not make sense as it is
and thus one needs to introduce  renormalization.
See Remark \ref{REM:rough} below.
In this paper, we focus 
on the case $s > \frac d 2$
such that a typical element in the support of $\mu_s$ is a function.

The main difficulty in the construction of the focusing Gibbs measure $\rho_{s, p}$
in \eqref{Gibbs}
comes from the unboundedness of the potential energy.
In fact, with \eqref{random_s}
and $\l^2(\Z^d) \subset \l^p(\Z^d)$, we immediately see that 
\[ \E_{\mu_s} \bigg[ e^{\frac 1p (\int_{\T^d} |u|^2dx)^{\frac{p}2}}\bigg]
\geq \E_{\mu_s} \bigg[ e^{\frac 1p \|u\|_{\F L^p (\T^d)}^p}\bigg]
\geq \prod_{n \in \Z^d} \E\bigg[e^{\frac{|g_n|^p}{(2\pi |n|)^{sp}}}\bigg]
= \infty
\]

\noi
for $p > 2$,
where $\F L^p  (\T^d)$ denotes Fourier Lebesgue space defined by the norm
\[
\| u\|_{\F L^p  (\T^d)} = \bigg( \sum_{n\in \Z^d} |\ft u (n)|^p \bigg)^{\frac1p}
.
\]

\noi
In a seminal work \cite{LRS}, Lebowitz-Rose-Speer proposed to consider
the focusing Gibbs measure of the following form  (when $d = s = 1$):
\begin{align}
d\rho_{s,p} (u) = Z_{1,p}^{-1}
\ind_{\{\|u\|_{L^2(\T^d)}\le K\}}
 \exp \bigg(\frac1{p} \int_{\T^d} |u|^p dx \bigg) d\mu_s (u).
\label{Gibbs2}
\end{align}

\noi
In \cite{LRS, BO94, OST1}, 
Lebowitz-Rose-Speer, Bourgain, and Oh-Sosoe-Tolomeo showed that 
the Gibbs measure $\rho_{1, p}$ on the one-dimensional torus $\T$
 is indeed 
 normalizable for (i) $2 < p < 6$ and any finite $K > 0$
 and (ii) $p = 6$ and any $0 < K \leq \|Q\|_{L^2(\R)}$, 
 where $Q$ is the (unique) minimizer of 
 the Gagliardo-Nirenberg-Sobolev inequality on $\R$
 with $\|Q\|_{L^6(\R)}^6 = 3 \|D Q\|_{L^2(\R)}$, 
 while it is not normalizable 
 for (iii) $p > 6$ 
 and (iv) $p = 6$ and $K > \|Q\|_{L^2(\R)}$.

The main purpose of this paper is to study the focusing Gibbs measure
$\rho_{s,p}$ in \eqref{Gibbs2} with a mass cutoff for $s > \frac{d}2$ and to identify sharp conditions for its normalizability.
More precisely,
in the subcritical case (i.e.~$2 < p < \frac {4s}{d} + 2$), 
we prove that the focusing Gibbs measure $\rho_{s,p}$ in~\eqref{Gibbs2} 
is normalizable for any $K > 0$, 
while we prove its non-normalizability for any $K > 0$ in the supercritical case
(i.e.~$p> \frac {4s}{d} + 2$).
In the critical case (i.e.~$p = \frac {4s}{d} + 2$), 
we then show that 
there is a critical mass threshold, 
characterized by an optimizer for the Gagliardo-Nirenberg-Sobolev inequality on $\R^d$
(see Section~\ref{SEC:GNS})
such that  the focusing Gibbs measure $\rho_{s,p}$ 
is normalizable below this critical mass threshold, while it is not above this threshold.
See Theorem~\ref{THM:main}.
Previously, these results were known only for $d = s = 1$  (\cite{LRS,BO94,OST1})
and our aim is to extend these results for any $d \geq 1$, $s > 0$, 
and $p > 2$, provided that $s > \frac d2$ (such that the random Fourier series 
$u$ in \eqref{random_s} defines a function).
See \cite{OST} for the case $ s= \frac d2$.
We refer the reader to \cite{OST1} and the references therein for further  background.

The energy functional associated with   the Gibbs measure $\rho_{s,p}$ in \eqref{Gibbs2} is given by
\begin{align}
    \label{Hamiltonian}
    H_{\T^d}(u) = \frac12 \int_{\T^d} |D^{s} u|^2 d x - \frac1p \int_{\T^d} |u|^p d x.
\end{align}

\noi
We point out that the construction of the Gibbs measure
is not only of interest in the area of mathematical physics such as 
constructive Euclidean quantum field theory,
but is also crucial in the study of Hamiltonian PDEs \cite{LRS, BO94,BO96,OQV,ORTh,OOT, OST,OST1}. 
An important example of Hamiltonian PDEs corresponding to
the energy functional \eqref{Hamiltonian}
is the following fractional nonlinear Schr\"odinger equation:
\begin{align}
i\dt u + D^{2s} u = |u|^{p-2} u.
\label{fNLS}
\end{align}


\noi
The equation \eqref{fNLS} corresponds
to the nonlinear Schr\"odinger equation (NLS) when $s=1$ (\cite{BO94,BO96}),
 to the biharmonic NLS  when $s =2$ (\cite{OTz, OSTz, OTW}), 
 and to the nonlinear half-wave 
equation when $s = \frac12$ (\cite{Poco}).
In the seminal work \cite{BO94,BO96},
Bourgain showed that we can extend the local-in-time dynamics of \eqref{fNLS}
globally in time by using the 
Gibbs measure\footnote{Strictly speaking, we need to modify 
the massless fractional Gaussian free field $\mu_s$ in \eqref{GPM}
by the massive one to avoid an issue at the zeroth frequency.  See Remark \ref{REM:mass}.
}
 as a replacement of a conservation law.
Over the last decade, we have seen a tremendous progress
in the study of this subject.
See \cite{BOP4} for a survey on the subject and for the references therein.

\medskip

We now state our main result.

\begin{theorem}\label{THM:main}

Let  $d \geq 1$, $s>\frac{d}2$,  and  $p > 2$.
Given  $K > 0$, define
the partition function $Z_{s,p, K}$ by 
\begin{align}
\label{Z1}
Z_{s,p,K}=\E_{\mu_s}
\Big[e^{\frac{1}{p} \int_{\T^d} |u|^p\,\mathd  x}\ind_{\{\|u\|_{L^2(\T^d)}\le K\}}\Big], 
\end{align}

\noi
where $\E_{\mu_s}$ denotes an expectation with respect
to the law $\mu_s$ of the random Fourier series in  
  \eqref{random_s}.
Then, the following statements hold\textup{:}

\smallskip

\noi
\begin{itemize}
\item
[\textup{(i)}] \textup{(subcritical case)}
If $2< p< \frac{4s}d+2$, then $Z_{s,p,K}<\infty$ for any $K>0$.

\smallskip

\noi
\item[\textup{(ii)}] \textup{(critical case)}
Let $p=\frac{4s}d+2$.
Then, $Z_{s,p,K}<\infty$ if $K<\|Q\|_{L^2(\R^d)}$, and $Z_{s,p,K}=\infty$ if $K>\|Q\|_{L^2(\R^d)}$. Here, $Q$ is the optimizer for the Gagliardo-Nirenberg-Sobolev inequality
on $\R^d$ such that $\| Q\|_{L^p (\R^d)}^p = \frac{p}2 \|D^s Q\|_{L^2 (\R^d)}^2$.

\item[\textup{(iii)}]
\textup{(supercritical case)}
If $ p> \frac{4s}d+2$,
then $Z_{s,p,K} = \infty$ for any $K >0$.

\end{itemize}

\end{theorem}

As mentioned above, 
Theorem \ref{THM:main} extends the results
in \cite{LRS, BO94, OST1}
to  $d\geq 1$ and $s > \frac d2 $.
When $d = s= 1$, 
Bourgain \cite{BO94} proved Theorem \ref{THM:main}\,(i)
and also (ii) (but with sufficiently small $K \ll 1$), 
using the dyadic pigeon hole principle
and the Sobolev embedding theorem.
In \cite{LRS}, Lebowitz-Rose-Speer  proved the non-normalizability 
in Theorem~\ref{THM:main} (i.e.~for $K > \|Q\|_{L^2 (\R)}$ when $p = 6$ 
 and for any $K >0$ when $p > 6$).
Their argument was based on  a Cameron-Martin type 
argument and the sharp Gagliardo-Nirenberg-Sobolev (GNS) inequality:
\begin{align}
\label{GNS}
\|u\|_{L^p(\R)}^{p} \le C_{\textup{GNS}} \| u\|^{\frac{p}{2}-1}_{\dot H^1 (\R)} \|u\|^{\frac{p}{2}+1}_{L^2(\R)},
\end{align}

\noi
where $C_{\textup{GNS}}$ is the optimal constant.
In \cite{OST1}, Oh-Sosoe-Tolomeo  refined Bourgain's argument 
and used the sharp GNS inequality
to prove the normalizability part in Theorem \ref{THM:main}\,(ii), 
thus
identifying the optimal mass threshold at the critical mass nonlinearity.
In the same paper, 
Oh-Sosoe-Tolomeo also proved 
the normalizability of the focusing Gibbs measure at the critical mass threshold
 $K=\|Q\|_{L^2(\R)}$
when $ d= s = 1$. See Remark \ref{REM:OST}.

The main difficulties in proving Theorem \ref{THM:main} 
come from the non-local nature of the fractional derivatives $D^s = (-\Dl)^\frac{s}{2}$
and the non-integer critical exponents $p = \frac{4s}d+2$.
In particular,
the non-local derivative poses extra difficulty in
localizing the GNS inequality, initially on $\R^d$, to the torus $\T^d$.
Inspired by \cite{BO},
we exploit a characterization of the $\dot H^s (\R^d)$-norm in terms of high order difference operators \eqref{difference}.
We then establish
 an almost sharp GNS inequality on $\T^d$ 
 (with the sharp constant $C_{\textup{GNS}}$ in \eqref{GNS})
 by using this new characterization. 
See Proposition \ref{PROP:T}.

With the sharp GNS inequality on $\T^d$, 
our proof of   Theorem \ref{THM:main} is based on 
 the variational approach due
to Barashkov-Gubinelli~\cite{BG} 
and  is quite different from those in \cite{LRS, BO94, OST1}, 
thus 
providing  an alternative   proof of the results when $d = s = 1$.
We first 
express the partition function 
$Z_{s,p,K}$ in \eqref{Z1} in a stochastic optimization problem, 
using 
the Bou\'e-Dupuis variational formula (Lemma \ref{LEM:var}).
We then prove 
 the normalizability part 
 of Theorem \ref{THM:main}, 
by using the almost sharp GNS inequality on $\T^d$.
As for the non-normalizability part,
the main task is to construct a sequence of drift terms
which achieves the divergence of the partition function. 
Our construction of such drift terms
is based on a scaling argument, 
analogous to that in \cite{LRS, OST1}.
We point out that our proof of Theorem \ref{THM:main}, based on the  variational approach, 
is essentially a physical space approach, instead of the Fourier side approach in \cite{BO94, OST1}.
It is thus expected that our approach is more flexible in geometric settings.
See also Appendix B in \cite{OST}, 
where the variational approach was used
to prove Theorem \ref{THM:main}\,(i)
with $s > \frac d2$ and $p = 4$.

\begin{remark}\rm
The key idea in proving Theorem \ref{THM:main} lies in controlling the potential energy $\frac1p \| u \|_{L^p(\T^d)}^p$ by the kinetic energy $\frac12 \|u\|^2_{\dot H^s(\T^d)}$ under the constraint $\| u \|_{L^2 (\T^d)} \le K$.
From Gagliardo-Nirenberg-Sobolev inequality Proposition \ref{PROP:T},
we see that the subcritical case $2 < p < \frac{4s}d +2$ corresponds to weaker potential energy.
The critical exponent $p = \frac{4s}d +2$ leads to the equivalence of potential and kinetic energy, 
where a restriction on the size $K$ is needed to guarantee the normalizability.
For the supercritical case  $p > \frac{4s}d +2$, however,
the kinetic energy losses control of the potential energy no matter how small the mass is.
\end{remark}

\begin{remark}\rm \label{REM:mass}
As in \cite{OST1},
Theorem \ref{THM:main} also applies
 when we replace
the mean-zero fractional Brownian loop in~\eqref{random_s}
by 
 the fractional Ornstein-Uhlenbeck loop: 
\begin{equation}
\label{bloop3}
u(x)=\sum_{n \in \Z^d}\frac{g_n(\o)}{\jb{n}^s}e^{2\pi inx},
\end{equation}

\noi
where $\jb{n} = (1 + 4\pi^2 |n|^2)^\frac{1}{2}$
and  $\{g_n\}_{n \in \Z^d}$ is a sequence of independent standard complex-valued  Gaussian random variables.
See Remark 4.1 in \cite{OST1}.
The law $\wt \mu_s$ of  the fractional Ornstein-Uhlenbeck loop in \eqref{bloop3} has 
the formal density 
\begin{align*}
d \wt \mu_s = \wt Z_s^{-1} e^{-\frac 12 \|u\|_{H^s(\T^d)}^2} d u .
\end{align*}

\noi
As seen in \cite{BO94}, 
the measure $\wt \mu_s$ is a more natural base Gaussian measure to consider for
the (fractional) nonlinear Schr\"odinger equation \eqref{fNLS}
due to the lack of the conservation of the spatial mean under the dynamics.

Note  that 
Theorem \ref{THM:main} also holds in the real-valued setting
(i.e.~with an extra assumption that $g_{-n} = \cj g_n$ in \eqref{random_s}).
For example, this  is relevant to the study of the dispersion generalized KdV equation
on $\T$:
\begin{align*}
\dt u + D^{2s} \dx u =  \dx (u^{p-1}).
\end{align*}

\end{remark}

\begin{remark}\label{REM:OST} \rm
We point out  that Oh-Sosoe-Tolomeo \cite{OST1}
also showed the normalizability of the Gibbs measure \eqref{Gibbs} at the critical mass threshold when $s=d=1$ and $p=6$.
This result
 is quite striking in view of the presence of the minimal mass blowup solution 
 (at this critical mass) for the focusing quintic NLS on $\T$. 
We will not pursue this question for the fractional focusing Gibbs measure \eqref{Gibbs2},
as their argument is beyond the scope of the framework developed in this paper.
\end{remark}

\begin{remark} \label{REM:rough}\rm 
Since $s > \frac{d}2$,
Theorem \ref{THM:main} only considers the non-singular case,
namely, the measures $\mu_s$ and $\rho_{s, p}$ are supported on functions.
One of the reasons for  only considering the non-singular case 
is that 
the bifurcation phenomena 
at the critical mass (Theorem \ref{THM:main}\,(ii))
are only possible when $s > \frac{d}2$.  
As soon as $s \leq \frac d2$, 
we need to introduce a  proper renormalization to define the potential energy
 $\frac{1}{p} \int_{\T^d} |u|^p\,  d  x$, 
 which necessitates $p$ to be an integer.
When $s = \frac d2$, 
it was shown in \cite{BS,OST} that the renormalized 
focusing Gibbs measure $\rho_{\frac{d}2,4}$ (with $p = 4$, critical), 
endowed with a (renormalized) mass cutoff,
is not normalizable.
It was also shown in \cite{OST} that 
with the cubic interaction ($p=3$, subcritical), 
the renormalized focusing Gibbs measure  $\rho_{\frac{d}2,3}$
endowed with a renormalized mass cutoff
is indeed normalizable.
When $ d= 2$, this normalizability in the case of  the cubic interaction was first observed by Bourgain~\cite{BO99}.
When $d=3$, it has recently been shown that the cubic interaction ($ p = 3$) exhibits
phase transition between weakly and strongly nonlinear regimes. See \cite{OOT2} for more details.

\end{remark}

\begin{remark}\rm

While the construction of the defocusing Gibbs measures
has been extensively studied
and well understood
due to the strong interest in constructive Euclidean quantum field theory
(see \cite{Simon,GJ,ST}), 
the (non-)normalizability issue of the focusing Gibbs measures, 
going back  to the work of Lebowitz-Rose-Speer \cite{LRS}
and Brydges-Slade \cite{BS}, 
is not fully explored.
See   related works
\cite{Rider, BBulut, CFL, OOT, OST1, OST, OOT2, TW}
on the non-normalizability (and other issues)
for focusing Gibbs measures.
In particular, recent works such as \cite{OOT, OOT2}
employ
the variational approach developed in \cite{BG}
and establish certain phase transition
phenomena in the singular setting.


\end{remark}


\section{Sharp Gagliardo-Nirenberg-Sobolev inequality}
\label{SEC:GNS}

In order to  prove Theorem \ref{THM:main},
we need the sharp Gagliardo-Nirenberg-Sobolev inequality on $\T^d$.
We first recall the definition of the homogeneous Sobolev space
$\dot H^s(\R^d)$ defined by the norm:
\begin{align}
\label{HsR}
    \| u\|_{\dot H^s (\R^d)}^2 =   \int_{\R^d} (2\pi |\xi|)^{2s} |\ft u(\xi)|^2 d\xi.
\end{align}




\noi 
As mentioned in Section \ref{SEC:1}, 
the optimizer for the Gagliardo-Nirenberg-Sobolev inequality
with the optimal constant:
\begin{equation}\label{GNSdisp}
\|u\|_{L^{p}(\R^d)}^p \le C_{\textup{GNS}}(d,p,s)\| u\|^{\frac{(p-2)d}{2s}}_{\dot H^s (\R^d)}\|u\|^{2+\frac{p-2}{2s}(2s-d)}_{L^2(\R^d)}
\end{equation}

\noi
plays an important role in the study of the focusing Gibbs measures.
We recall the following result,

\begin{theorem}[Theorem 2.1, \cite{BFV14}]
\label{THM:fGNS}
Let $d \geq 1$ and let \textup{(i)} $p > 2$ if $d < 2s$, 
and \textup{(ii)} $2 < p \le \frac{2d}{d-2s}$ if $d \geq 2s$.
Consider the functional
\begin{align}
J^{d,p,s}(u)= \frac{\| u\|^{\frac{(p-2)d}{2s}}_{\dot H^s (\R^d)} \|u\|^{2+\frac{p-2}{2s}(2s-d)}_{L^2(\R^d)}}{\|u\|_{L^p(\R^d)}^{p}}
\label{JJ1}
\end{align}

\noi
on $H^s(\R^d)$. 
Then, the minimum
\begin{align}
C^{-1}_{\GNS} = C_{\GNS}(d,p,s)^{-1}:= \inf_{\substack{u\in H^s(\R^d)\\u \ne 0}} J^{d, p,s}(u)
\label{CGNS}
\end{align}

\noi
is attained 
at a function  $Q\in H^s (\R^d) $.
\end{theorem}

\begin{remark}\rm
\label{RMK:scaling}
It is easy to see that functions $u (x): = c Q(b(x-a))$ for all $c \in \R \backslash \{0\}$, $b >0$, and $a \in \R^d$, are minimizers of the functional 
\eqref{JJ1}.
Therefore, 
we may assume that
\begin{align}
\label{scaling}
\begin{split}
\| Q \|_{L^2 (\R^d)} & = \| Q \|_{\dot H^s (\R^d)}, \\ 
\| Q\|^2_{\dot H^s (\R^d)} & = \frac2p \|Q\|_{L^p(\R^d)}^p .
\end{split}
\end{align}

\noi
Under this specified scaling, we have $H_{\R^d}(Q) = 0$,
where $H_{\R^d}$ is the Hamiltonian functional given 
in \eqref{Hamiltonian} with $\T^d$ being replaced by $\R^d$. 
Furthermore, 
this $Q$ solves the following semilinear elliptic equation on $\R^d$:
\begin{align}\label{elliptic}
 (p-2) d D^{2s} Q + (4s + (p-2)(2s-d)) Q - 4s Q^{p-1} = 0 .
\end{align}

\noi
In the following,
we restrict ourselves to \eqref{scaling} 
unless specified otherwise. 
In particular, we have
\begin{align}
\label{CGNS1}
C_{\textup{GNS}} = \frac{p}2 \|Q\|_{L^2 (\R^d)}^{2-p}.
\end{align}

\noi
The uniqueness  \textup{(}in some sense\textup{)} of this $Q$ for fractional value $s$
is a very challenging problem,
which is only proved for some special cases, for instance when $d=1$ and $s\in (0,1]$. See \cite{FrankL}.
\end{remark}


\medskip

For a function $u$ defined on $\T^d$,
we define the $\dot H^s (\T^d)$ norm via
\begin{align}
\label{HsT}
 \| u\|_{\dot H^s (\T^d)}^2 = \sum_{n \in \Z^d \backslash \{0\}} |n|^{2s} |\ft u(n)|^2. 
\end{align}

\noi
Due to the scaling invariance of the minimization problem \eqref{JJ1},
it is expected that the GNS inequality \eqref{GNSdisp} also holds on the finite domains $\T^d$ with the same optimal constants.

\begin{proposition}
\label{PROP:T}
Let $d \geq 1$ and let \textup{(i)} $p > 2$ if $d < 2s$, 
and \textup{(ii)} $2 < p \le \frac{2d}{d-2s}$ if $d \geq 2s$.
Then, given small $\delta>0$,
there is a constant $C  = C(\dl) >0$ such that
\begin{align}
{\|u\|_{L^p(\T^d)}^{p}} \le (C_{\GNS}(d,p,s) +  \dl) \|u\|^{\frac{(p-2)d}{2s}}_{\dot H^s (\T^d)}\|u\|^{2+\frac{p-2}{2s}(2s-d)}_{L^2(\T^d)} + C(\dl) {\|u\|_{L^2(\T^d)}^{p}} 
\label{GNST}
\end{align}

\noi
for $u \in H^s(\T^d)$, where $C_{\GNS}$ is the constant defined in \eqref{CGNS}
and \eqref{CGNS1}.
\end{proposition}

The main difficulty in showing  Proposition \ref{PROP:T} is due to the non-local nature of the fractional derivatives.
To circumvent this difficulty, 
we recall the characterization of the $\dot H^s (\R^d)$ norm 
\eqref{HsR} based on  the $L^2$-modulus of continuity. 
When $0<s <1$, one has
\begin{align}
\label{modulus}
\begin{split}
\int_{\R^d}\int_{\R^d} \frac{|u(x) -u(y)|^2}{|x-y|^{d+2s}} dxdy & = 
\int_{\R^d}\int_{\R^d} \frac{|u(x+y) -u(x)|^2}{|y|^{d+2s}} dxdy\\
& = \int_{\R^d} \bigg( |\xi|^{-2s} \int_{\R^d} \frac{|e^{2\pi i y \cdot \xi } -1|^2}{|y|^{d+2s}} dy \bigg)  |\xi|^{2s} |\ft u (\xi)|^2d \xi.
\end{split}
\end{align}

\noi
Denote the inner integral, 
a convergent improper integral for $0<s<1$, 
by
\begin{align}
\label{c_1}
c_{1}(d,s) = |\xi|^{-2s} \int_{\R^d} \frac{|e^{2\pi i y \cdot \xi } -1|^2}{|y|^{d+2s}} dy \,\,\,\bigg(= \int_{\R^d} \frac{|e^{2\pi i x_1 } -1|^2}{|x|^{d+2s}} dx\bigg),
\end{align}

\noi
which is a constant, i.e. independent of $\xi$.
From \eqref{modulus} and \eqref{c_1}, 
we have the following characterization of 
the $\dot H^s (\R^d)$ norm in \eqref{HsR} (see for example \cite{BO}), 
\begin{align}
    \| u\|_{\dot H^s (\R^d)}^2 =   c_1 (d,s)^{-1} \int_{\R^d}\int_{\R^d} \frac{|u(x) -u(y)|^2}{|x-y|^{d+2s}} dxdy.
    \label{Horm}
\end{align}

\noi
We remark that on the torus $\T^d$ the $H^s (\T^d)$ norm defined in 
\eqref{HsT} has a similar equivalent characterization. See \cite[Proposition 1.3]{BO}.
However, the identity as \eqref{Horm} fails for the torus case 
due to the lack of rotational invariance.


By using high order difference operators,
we may generalize \eqref{Horm} to the cases $s\ge 1$. 
In particular, we have 
\begin{align}
    \| u\|_{\dot H^s (\R^d)}^2 =   c_k (d,s)^{-1} \int_{\R^d}\int_{\R^d} \frac{|\Delta_y^k u(x)|^2}{|y|^{d+2s}} dxdy,
    \label{Horm2}
\end{align}

\noi
where $\Delta_y^k$ is the $k$-th forward difference operator with spacing $y$ defined by
\begin{align}
\label{difference}
\Delta_y^k u(x) = \sum_{j=0}^k (-1)^{k-j} C_k^j u(x+jy),
\end{align}

\noi
where $C_k^j$ are binomial coefficients, and
\begin{align}
    \label{c_k}
c_k(d,s) = \int_{\R^d} \frac{|e^{2\pi i x_1 } -1|^{2k}}{|x|^{d+2s}} dx.
\end{align}

\noi
The proof of \eqref{Horm2} is similar to that of \eqref{Horm}. We thus omit the details.

\medskip

Now we are ready to prove Proposition \ref{PROP:T}.






\begin{proof}[Proof of Proposition \ref{PROP:T}]
For the pedagogical purpose,
we present the proof for the case $0<s<1$
before demonstrating the general case $s > 0$,
as the former is less complex in terms of notation.

We first consider $0<s<1$. 
Let $\psi \in C_0^\infty (B(0,\frac12))$ be a bump function with $\|\psi\|_{L^1} = 1$
and $\psi_\delta (x) = \dl^{-d} \psi(\frac{x}\dl)$.
Define $\phi_\dl (x) = \ind_{[-\frac12 + 2\dl, \frac12 - 2\dl]^d} * \psi_\dl (x)$.
Then the following properties hold 
\begin{itemize}
    \item[(i)] $\phi_\dl \in C^\infty_0 (\T^d)$,
    \item[(ii)] $\phi_\dl(x) = 1$ for $x\in [-\frac12 + 3\dl, \frac12 - 3\dl]^d$,
    \item[(iii)] $\phi_\dl(x) = 0$ for $x\in ([-\frac12 + \dl, \frac12 - \dl]^d)^c$,
    \item[(iv)] $|D^s \phi_\dl (x)| \les \dl^{-s}$ for all $x \in \R^d$.
\end{itemize} 
Let
\begin{align}
\label{extension}
u_\delta (x) = 
\begin{cases}
\phi_\delta (x) u (x), & x \in [-\frac12,\frac12]^d;\\
0, & \textup{otherwise}.
\end{cases}
\end{align}

First we claim there exists $C(d) >0$ such that for any $u \in L^p(\T^d)$ there exists $x_0 \in \T^d$ satisfying the following 
\begin{align}
\| u \|_{L^p (\T^d)}^p \le (1+ C (d) \dl) \| \phi_\dl (\cdot) u(\cdot + x_0)  \|_{L^p (\R^d)}^p .
\label{recenter0}
\end{align}

\noi
From the definition of $\phi_\dl$, it suffices to show
\begin{align}
\| u \|_{L^p (\T^d)}^p \le (1+ C (d) \dl) \| u(\cdot + x_0)  \|_{L^p ( [-\frac12 + 3\dl, \frac12 - 3\dl]^d)}^p .
\label{recenter}
\end{align}

\noi
We show \eqref{recenter} inductively. 
Recall that $\dl \ll 1$.
When $d = 1$, we may split the interval $[-\frac12, \frac12]$ into $k = [\frac1{6\dl}]$ many equal subintervals. 
Then, from the pigeonhole principle, there must be a subinterval, say the $j$-th subinterval $[-\frac12 + \frac{j-1}k,  -\frac12 + \frac{j}k]$, such that
\[
\int_{[-\frac12 + \frac{j-1}k,  -\frac12 + \frac{j}k]} |u (x )|^p dx \le \frac1{k} \int_{\T} |u(x)|^p dx,
\] 

\noi
which implies 
\[
\begin{split}
\int_{\T} |u(x)|^p dx & \le (1+ \frac1k) \int_{-\frac12 + \frac1{2k}}^{\frac12 - \frac1{2k}} \Big|u\Big(x + \frac{2j-1}{2k}\Big) \Big|^p dx\\
& \le (1+ 12 \dl) \int_{-\frac12 + 3\dl}^{\frac12 - 3\dl} \Big|u\Big(x + \frac{2j-1}{2k}\Big) \Big|^p dx,
\end{split}
\]

\noi
provided $\dl$ is sufficiently small.
Thus we conclude \eqref{recenter} for $d=1$.
Let us assume \eqref{recenter} holds for all $1,2,\cdots, d-1$ dimensions.
Then for $x \in \T^d$, we may write $x = (x', x_d)$ such that $x' \in \T^{d-1}$ and $x_d \in \T$. 
Then, from our assumption, 
there exist $x_d^0 \in \T$ and $x'_0 \in \T^{d-1}$ such that
\[
\begin{split}
\int_{\T^d} & |u(x)|^p dx  = \int_\T \Big( \int_{\T^{d-1}} |u(x',x_d)|^p dx' \Big) dx_d\\
& \le (1+ C\dl) \int_{-\frac12 + 3\dl}^{\frac12 - 3\dl}  \Big( \int_{\T^{d-1}} |u(x',x_d + x_d^0)|^p dx' \Big) dx_d\\
& \le (1+ C\dl)  \int_{\T^{d-1}}  \Big( \int_{-\frac12 + 3\dl}^{\frac12 - 3\dl} |u(x',x_d + x_d^0)|^p dx_d \Big) dx'  \\
& \le (1+ C (1) \dl) \big(1+ C (d-1) \dl \big)    \int_{[-\frac12 + 3\dl, \frac12 - 3\dl]^{d-1}}  \Big( \int_{-\frac12 + 3\dl}^{\frac12 - 3\dl} |u(x' + x'_0,x_d + x_d^0)|^p dx_d \Big) dx' ,
\end{split}
\]

\noi
where we used the assumption in the second and fourth steps.
Thus we finish the proof of \eqref{recenter} for $d$ dimension by taking $x_0 = (x_0',x_d^0)$ and $C(d) = 1+ (C(1) + 2 C(d-1))$ provided $C(1) \dl < 1$.

From \eqref{recenter0}, for the translated $u (\cdot + x_0)$,
still denoting by $u$\footnote{We note that \eqref{GNST} is invariant under translation.},
we have
\begin{align}
    \|&u\|^p_{L^p (\T^d)}  \le  (1+ C \dl) \| u_\dl\|^p_{L^p (\R^d)} \notag\\ 
    & \le (1+ C \dl)  C_{\textup{GNS}}(d,p,s)\| u_\dl \|^{\frac{(p-2)d}{2s}}_{\dot H^s (\R^d)}\|u_\dl\|^{2 + \frac{p-2}{2s}(2s-d)}_{L^2(\R^d)} \notag\\
    & \le  (1+ C \dl)  C_{\textup{GNS}}(d,p,s)\bigg( c_1(d,s)^{-1} \int_{\R^d}\int_{\R^d} \frac{|u_\dl (x) - u_\dl (y)|^2}{|x-y|^{d+2s}} dxdy \bigg)^{\frac{(p-2)d}{4s}}\| u\|^{2+\frac{p-2}{2s}(2s-d)}_{L^2(\T^d)}. \notag
\end{align}

\noi
To prove \eqref{GNST}, 
it only needs to show
\begin{align}
    \label{GNSTa}
    \int_{\R^d}\int_{\R^d} \frac{|u_\dl (x) - u_\dl (y)|^2}{|x-y|^{d+2s}} dxdy
    \le (1+ C\delta ) c_1 (d,s) \|u\|^2_{H^s (\T^d)} + C(\dl) \|u\|^2_{L^2(\T^d)}.
\end{align}

\noi
Since the integrand $\frac{|u_\dl(x) - u_\dl (y)|^2}{|x-y|^{d+2s}}$ in \eqref{GNSTa} is supported on $(x,y)\in (\T^d \times \R^d) \cup (\R^d \times \T^d)$, we have  
\begin{align} 
\begin{split}
    \textup{LHS of } \eqref{GNSTa} & \le    \int_{\T^d}\int_{\T^d} \frac{| u_\dl (x) -u_\dl(y)|^2}{|x-y|^{d+2s}} dxdy  + C   \int_{(\T^d)^c}\int_{\T^d} \frac{| u_\dl (x)|^2}{|x-y|^{d+2s}} dxdy .
 \end{split}
 \label{GNSTb}
\end{align}
    
\noi
For the second term in \eqref{GNSTb},
since $|x-y| > \delta$ in the integrand, 
we have
\begin{align}
    \label{II}
\int_{(\T^d)^c}\int_{\T^d} \frac{| u_\dl (x)|^2}{|x-y|^{d+2s}} dxdy  \les \bigg(  \int_{|y|>\delta} \frac{1}{|y|^{d+2s}} dy  \bigg) 
\|u_\dl \|_{L^2(\T^d)}^2
\les \dl^{-2s} \|u\|_{L^2 (\T^d)}^2,
\end{align}

\noi
which is sufficient for \eqref{GNSTa}.
Now we turn to the first term in \eqref{GNSTb}.
We note
\begin{align*} 
         | u_\dl(x) -  &u_\dl(y)|^2 
         = |\phi_\dl (x) (u(x) - u(y)) + (\phi_\dl (x) - \phi_\dl (y)) u(y)|^2 \notag\\
        & = |\phi_\dl (x) (u(x) - u(y))|^2 +  |(\phi_\dl (x) - \phi_\dl (y)) u(y)|^2 \\
        & \hphantom{XX} + 2 \phi_\dl (x)  (\phi_\dl (x) - \phi_\dl (y)) (u(x) - u(y)) u(y).
\end{align*}

\noi
Thus we have
\begin{align}
    \begin{split}
         & \int_{\T^d}\int_{\T^d} \frac{| u_\dl (x) -u_\dl(y)|^2}{|x-y|^{d+2s}} dxdy  
           \le  \int_{\T^d}\int_{\T^d} \frac{| u(x) -u(y)|^2}{|x-y|^{d+2s}} dxdy, \\
         & \hphantom{XX}+  \int_{\T^d}\int_{\T^d} \frac{|\phi_\dl (x)  (\phi_\dl (x) - \phi_\dl (y)) (u(x) - u(y)) u(y)|}{|x-y|^{d+2s}} dxdy, \\
         & \hphantom{XX}+ \int_{\T^d}\int_{\T^d} \frac{ |(\phi_\dl (x) - \phi_\dl (y)) u(y)|^2}{|x-y|^{d+2s}} dxdy \\
         & = A_1 + A_2 + A_3.
    \end{split}
    \label{A}
\end{align}

\noi
For the term $A_1$,
we have
 \begin{align}
 \begin{split} 
 A_1 & \le \int_{\T^d}\int_{B(0,2)} \frac{|u(x) -u(x+z)|^2}{|z|^{d+2s}} dzdx\\
 & = \sum_{n \in \Z^d \backslash \{0\}} 
 \bigg( |n|^{-2s} \int_{B(0,2)} \frac{|e^{2\pi i x \cdot n} -1|^2}{|x|^{d+2s}} dx\bigg)
 |n|^{2s}  |\ft u (n)|^2,
 \end{split}
\label{A1}
\end{align}

\noi
where $B(0,2) \subset \R^d$ is the ball centered at $0$ with radius $2$.
It is easy to see that  
\begin{align*}
 |n|^{-2s} \int_{B(0,2)} \frac{|e^{2\pi i x \cdot n} -1|^2}{|x|^{d+2s}} dx  \le  |n|^{-2s} \int_{\R^d} \frac{|e^{2\pi i x \cdot n} -1|^2}{|x|^{d+2s}} dx = {c_1(d,s)},
\end{align*}

\noi
which together with  \eqref{A1} shows the contribution from $A_1$ is bounded by the right hand side of \eqref{GNSTa}.
For the term $A_3$, we have
\begin{align}
    \begin{split}
        A_3 \les \delta^{-2}\int_{\T^d}\int_{\T^d} \frac{   |u(y)|^2}{|x-y|^{d+2s-2}} dxdy \les \delta^{-2} \|u\|^2_{L^2 (\T^d)},
    \end{split}\label{A3}
\end{align}

\noi
which is sufficient for our purpose.
For $A_2$, by Young's inequality we have
\begin{align} 
\begin{split} 
        A_2 & \le \delta A_1 + \frac1\delta A_3,
    \end{split}
    \label{A2}
\end{align}

\noi
which is again acceptable.
By collecting 
\eqref{A}, \eqref{A1}, \eqref{A2}, and \eqref{A3},  we finish the proof of \eqref{GNSTa} and thus \eqref{GNST} when $0<s<1$.

In the following we consider the case $s\ge 1$.
Assume $s \in [k-1, k)$ for some $k\in \Z_+$.
Similarly to \eqref{GNSTa} in the case $0<s<1$,
it only needs to show
\begin{align}
    \label{generalka}
    \int_{\R^d}\int_{\R^d} \frac{| \Delta_y^k u_\dl (x)|^2}{|y|^{d+2s}} dydx \le (1+ C\delta ) c_k (d,s) \|u\|_{H^s (\T^d)}^2 + C(\dl) \|u\|_{L^2(\T^d)}^2.
\end{align}

\noi
Similarly to \eqref{II} and \eqref{A1},
we may reduce \eqref{generalka} to
\begin{align}
    \label{generalk}
    \int_{\T^d}\int_{B(0,k)} \frac{| \Delta_y^k u_\dl (x)|^2}{|y|^{d+2s}} dydx \le (1+ C\delta ) c_k (d,s) \|u\|_{H^s (\T^d)}^2 + C(\dl) \|u\|_{L^2(\T^d)}^2.
\end{align}

\noi
In the following, we prove \eqref{generalk}.
First note that
\begin{align*}
    \Delta_y^k u_\dl (x)  = \Delta_y^k ( \psi_\dl (x) u (x))
     = \sum_{j=0}^k C_k^j\Delta_y^{k-j} \psi_\dl (x) \Delta_y^{j} u(x+(k-j)y) .
\end{align*}

\noi
Therefore, we have
\begin{align}
    \label{generalk1}
    \begin{split}
    & \int_{\T^d}\int_{B(0,k)} \frac{| \Delta_y^k u_\dl (x)|^2}{|y|^{d+2s}}  dydx \\
    &  = \int_{\T^d}\int_{B(0,k)} \frac{| \sum_{j=0}^k C_k^j\Delta_y^{k-j} \psi_\dl (x) \Delta_y^{j} u(x+(k-j)y) |^2}{|y|^{d+2s}} dydx\\
    & = \int_{\T^d}\int_{B(0,k)} \frac{ |\psi_\dl (x) \Delta_y^{k} u(x) |^2}{|y|^{d+2s}} dydx \\ & \hphantom{} + \sum_{j=0}^{k-1} \int_{\T^d}\int_{B(0,k)} \frac{ |C_k^j\Delta_y^{k-j} \psi_\dl (x) \Delta_y^{j} u(x+(k-j)y) |^2}{|y|^{d+2s}} dydx \\
    & \hphantom{} + \sum_{j\neq \l} \int_{\T^d}\int_{B(0,k)} \frac{ C_k^j\Delta_y^{k-j} \psi_\dl (x) \Delta_y^{j} u(x+(k-j)y) C_k^\l\Delta_y^{k-\l} \psi_\dl (x) \Delta_y^{\l} u(x+(k- \l)y)}{|y|^{d+2s}} dydx \\
   & = B_1 + B_2 + B_3.
    \end{split}
\end{align}

\noi
For the term $B_1$ in \eqref{generalk1}, we have
\begin{align}
    \begin{split}
        B_1 & \le \int_{\T^d}\int_{B(0,k)} \frac{| \Delta_y^{k} u(x) |^2}{|y|^{d+2s}} dy dx\\
        & =  \sum_{n\in \Z^d}  \int_{B(0,k)} \frac{ |e^{2\pi i y\cdot n} - 1|^{2k} }{|y|^{d+2s}} dy |\ft u(n) |^2\\
        & \le c_k (d,s) \sum_{n\in \Z^d}   |n|^{2s}|\ft u(n) |^2,
    \end{split}
    \label{B1}
\end{align}

\noi
where $c_k (d,s)$ is defined in \eqref{c_k}.
Thus the contribution of $B_1$ is bounded by the right hand side of \eqref{generalk}.
Similarly, we can control $B_2$ in \eqref{generalk1} as
\begin{align}
    \begin{split}
        B_2 & \les \sum_{j=0}^{k-1} \int_{\T^d}\int_{B (0,k)} \frac{ |\Delta_y^{k-j} \psi_\dl (x) \Delta_y^{j} u(x+(k-j)y) |^2}{|y|^{d+2s}} dy dx\\
        & \les \sum_{j=0}^{k-1} \delta^{-2(k-j)} \int_{B (0,k)}  \int_{\T^d}  \frac{ | \Delta_y^{j} u(x+(k-j)y) |^2}{|y|^{d+2s-2(k-j)}}  dx dy\\
        & \les \sum_{j=0}^{k-1} \delta^{-2(k-j)}  \int_{\T^d} \int_{B (0,k)}  \frac{ | \Delta_y^{j} u(x) |^2}{|y|^{d+2s-2(k-j)}}  dy dx\\
        & \les \sum_{j=0}^{k-1} \delta^{-2(k-j)} \|u\|^2_{\dot H^{s-k+j} (\T^d)}\\
        & \les \dl \|u\|^2_{\dot H^{s}(\T^d)} + C(\dl) \|u\|^2_{L^2 (\T^d)},
    \end{split}
    \label{B2}
\end{align}

\noi
where in the last step we used the interpolation between $L^2(\T^d)$ and $\dot H^s (\T^d)$. This shows that the contribution of $B_2$ is acceptable.

Finally, we turn to $B_3$ in \eqref{generalk1}. When $j< k$ and $\l <k$, by H\"older's  inequality we have
\[
\int_{\T^d}\int_{\T^d} \frac{ C_k^j\Delta_y^{k-j} \psi_\dl (x) \Delta_y^{j} u(x+(k-j)y) C_k^\l\Delta_y^{k-\l} \psi_\dl (x) \Delta_y^{\l} u(x+(k-\l)y)}{|y|^{d+2s}} dxdy \les B_2,
\]

\noi
which is bounded by \eqref{B2}.
Without loss of generality, we only consider the case $j=k$. 
Then we have $\l < k$.
By Young's inequality we have
\[
\int_{\T^d}\int_{\T^d} \frac{ \psi_\dl (x) \Delta_y^{k} u(x) C_k^\l\Delta_y^{k-\l} \psi_\dl (x) \Delta_y^{\l} u(x+(k-\l)y)}{|y|^{d+2s}} dxdy \les \delta B_1 + C(\dl) B_2,
\]

\noi
which is again sufficient for our purpose
in view of \eqref{B1} and \eqref{B2}.

We finish the proof of \eqref{generalk},
and thus the proposition.
\end{proof}

\begin{remark}\rm
\label{RMK:compare}
Let $u$ be a function defined on $\R^d$.
With a slight abuse of notation,
we also use $u$ to denote its restriction onto $\T^d$. 
It follows from \eqref{Horm2} and \eqref{c_k} that
\begin{align}
\label{sob_com}
\begin{split}
\| u \|^2_{\dot H^s (\T^d)} & = c_k^{-1} \sum_{n\in \Z^d} \bigg( \int_{\R^d} \frac{|e^{2\pi i y \cdot n} -1|^{2k}}{|y|^{d+2s}} dy \bigg) |\ft u(n) |^2 \\
& = c_k^{-1}  \int_{\T^d}\int_{\R^d} \frac{| \Delta_y^{k} u(x) |^2}{|y|^{d+2s}} dydx \\
& \le c_k^{-1}  \int_{\R^d}\int_{\R^d} \frac{| \Delta_y^{k} u(x) |^2}{|y|^{d+2s}} dxdy  = \| u\|^2_{\dot H^s (\R^d)},
\end{split}
\end{align}

\noi
where $k = [s]$ is the largest integer less than $s$.

\end{remark}



\section{Proof of Theorem \ref{THM:main}}
\label{SEC:subcritical}

In this section, 
we prove Theorem \ref{THM:main},
which provides sharp criteria for the normalizability of the
Gibbs measure \eqref{Gibbs}
with focusing interaction.

\subsection{Variational formulation}
In order to prove Theorem \ref{THM:main},
we recall a variational formula for the partition functional 
$Z_{s,p,K}$ as in \cite{OST1}.
Let $W(t)$ denote a mean zero cylindrical Brownian motion in $L^2(\T^d)$
\begin{align*}
W(t) = \sum_{n\in \Z^d \backslash \{0\}} B_n (t) e_n
\end{align*}

\noi
where $\{B_n\}_{n \in \Z^d \backslash \{0\}}$ is a sequence of mutually independent complex-valued Brownian motions.
Then define a centered Gaussian process $Y_s(t)$ by
\begin{align}
\label{Yt}
Y_s(t) =  D^{-s} W(t) = \sum_{n\in \Z^d \backslash \{0\}} \frac{B_n (t)}{|n|^s} e_n.
\end{align}

\noi
We note that $Y_s(t)$ is well-defined and
\[
\E \big[|Y_s(1)|^2\big] = \sum_{n\in \Z^d \backslash \{0\}} \frac{\E [|B_n (1)|^2] }{|n|^s}  = \sum_{n\in \Z^d\backslash \{0\}} \frac{2}{|n|^s} < \infty,
\]

\noi
provided $s > \frac{d}2$.
In particular, 
we have 
\begin{align}
\label{law}
\textup{Law} (Y_s(1)) = \mu_s,
\end{align} 
where $\mu_s$ is the massless Gaussian free field given in \eqref{GPM}.

Let $\mathbb{H}_a$ be the space of drifts,
which consists of {\it mean zero} progressively measurable processes belonging to 
$L^2 ([0,1]; L^2 (\T^d))$,
$\P$-almost surely.
One of the key tools in this paper is the following Bou\'e-Dupuis variational formula \cite{BD,Ust,BD}.
See also \cite{BuD} for the infinite dimensional setting.

\begin{lemma}
\label{LEM:var}
Let $Y_s$ be as in \eqref{Yt} with $s> \frac{d}2$.
Suppose that $F: H^{s-\frac{d}2 -} (\T^d) \to \R$ is measurable
and bounded from above.
Then, we have
\begin{align}
\label{var}
-\log \E \Big[ e^{-F(Y_s(1))} \Big] = \inf_{\theta \in \mathbb{H}_a} \E
\bigg[F\big(Y_s(1) + I_s(\theta) (1) \big) + \frac12 \int_0^1 \| \theta (t) \|_{L^2_x}^2 dt\bigg],
\end{align}

\noi
where $I_s(\theta)$ is defined by
\[
I_s(\theta) (t) =   \int_0^t  D^{-s} P_{\neq 0} \theta (\tau) d\tau
\]

\noi
and
the expectation $\E = \E_\P$ is with respect to the underlying probability measure $\P$.
\end{lemma}

Since we only consider the non-singular case $s > \frac{d}2$,
then $Y_s(t)$ and $I_s(\theta)(1)$ enjoy
the following pathwise regularity bounds.

\begin{lemma}
\label{LEM:bounds}
\textup{(i)} Given any $s > \frac{d}2$ and any finite $p,q \ge 1$,
there exists $C_{s,p} >0$ such that
\begin{align}
\label{bounds1}
\E \big[ \|Y_s(1) \|_{L^q (\T^d)}^p \big] \le \E \big[ \|Y_s(1) \|_{L^\infty (\T^d)}^p \big] \le C_{s,p} < \infty .
\end{align}

\noi
\textup{(ii)} For any $\theta \in \mathbb{H}_a$,
we have
\begin{align}
\label{I}
\| I_s (\theta) (1)\|^2_{\dot H^s (\T^d)} \le \int_0^1 \| \theta (t) \|_{L^2_x}^2 dt.
\end{align}
\end{lemma}

\begin{proof}
Part (i) follows from 
H\"older's  inequality, 
Sobolev embedding,
Minkowski's inequality, and Wiener chaos estimate
\cite[Lemma 2.4]{BOP4} with $k=1$.
As for Part (ii), the estimate \eqref{I} follows from Minkowski's inequality and Cauchy-Schwarz' inequalities.
\end{proof}

%

We conclude this subsection by recalling 
the following simple corollary of  Fernique's theorem \cite{fernique}.
See also Theorem~2.7 in \cite{DZ}
and Lemma 4.2 in \cite{OST1}.

\begin{lemma}\label{LEM:Fer}
There exists a  constant $c>0$ such that if $X$ is a mean-zero Gaussian process  with values 
in a separable Banach space $B$ with $\E\big[\|X\|_{B}\big]<\infty$, then
\begin{align*}
\int e^{ c \frac{\|X\|_B^2}{(\E[\|X\|_B])^2}}\,d \P <\infty.
\end{align*}

\noi
In particular, we have
\begin{equation*}
\P\big(\|X\|_B \ge t \big)\les \exp \bigg[- \frac{c t^2}{ \big( \E\big[\|X\|_B\big] \big)^2}  \bigg]
\end{equation*}

\noi
for any $t>1$.
\end{lemma}

\subsection{Integrability}
In this subsection, 
we demonstrate the proof of 
the integrability part of Theorem \ref{THM:main}.
Namely,
we prove the boundedness of 
$Z_{s,p,K}$ (i) for all $K>0$ when $2< p < \frac{4s}d+2$ and 
(ii) for all $K < \|Q\|_{L^2(\R^d)}$ when $p=\frac{4s}d+2$,
where $Q$ is the optimizer for the GNS inequality on $\R^d$. 

\begin{proof}[Theorem \ref{THM:main} - {\rm (i)}
and the first half of {\rm (ii)}]
It suffices to show the following bound
\begin{align}
\label{var1}
Z_{s,p,K} = \E_{\mu_s}  \Big[ \exp (R_p(u)) \cdot  \ind_{\{\|u\|_{L^2(\T^d)}\le K\}}\Big] < \infty, 
\end{align}

\noi
where $R_p(u)$ is the potential energy denoted by
\begin{align}
\label{Rp}
R_p (u) : = {\frac{1}{p} \int_{\T^d} |u|^p\,\mathd  x}.
\end{align}

\noi
Observing that
\begin{align*}
\E_{\mu_s}  \Big[ \exp (R_p(u)) \cdot  \ind_{\{\|u\|_{L^2(\T^d)}\le K\}}\Big]
\le \E_{\mu_s}  \bigg[ \exp \Big(R_p(u) \cdot  \ind_{\{\|u\|_{L^2(\T^d)}\le K\}} \Big)\bigg],
\end{align*}

\noi
then the bound \eqref{var1} follows once we have
\begin{align}
\label{var3}
\begin{split}
&  \E_{\mu_s}  \bigg[ \exp \Big(R_p(u) \cdot  \ind_{\{\|u\|_{L^2(\T^d)}\le K\}} \Big)\bigg]  < \infty.
 \end{split}
\end{align}

\noi
From \eqref{law} and the Bou\'e-Dupuis variation formula Lemma \ref{LEM:var},
it follows that
\begin{align}
\label{var4}
\begin{split}
& -  \log  \E_{\mu_s}  \Big[ \exp \Big(R_p(u) \cdot  \ind_{\{\|u\|_{L^2(\T^d)}\le K\}} \Big)\Big]\\
& \hphantom{X}  = -  \log  \E  \Big[ \exp \Big(R_p(Y_s(1)) \cdot  \ind_{\{\|Y_s(1)\|_{L^2(\T^d)}\le K\}} \Big)\Big] \\
&\hphantom{X} = \inf_{\theta \in \mathbb{H}_a} \E \Big[ - R_p \big( Y_s(1) +I_s(\theta)(1) \big) 
\cdot  \ind_{\{\|Y_s(1) +I_s(\theta)(1)\|_{L^2(\T^d)}\le K\}} + \frac12 \int_0^1 
 \| \theta (t) \|_{L^2_x}^2 dt  \Big], 
\end{split}
\end{align}

\noi
where $Y(1)$ is given in \eqref{Yt}.
Here, $\E_{\mu_s}$ and $\E$ denote expectations with respect to the Gaussian field $\mu_s$ and the underlying probability measure $\P$ respectively.
In the following, 
we show that the right hand side of \eqref{var4} 
has a finite lower bound.
The key observation is that 
(i) in the subcritical setting, we view $Y_s(1)$ as a perturbation 
with finite $L^2(\T^d)$ norm;
(ii) in the critical setting, we have
``$Y_s(1) = P_{\le N} Y_s(1) +$ a perturbation"
for large $N \gg 1$, 
where the perturbation term is small under $L^2(\T^d)$
norm with large probability. 
We, therefore, distinguish two cases depending on subcritical/critical interactions.

\medskip

\noi
{\it \underline{Case 1}: subcritical $p <  \frac{4s}d +2$.} 
In this case, we prove \eqref{var3} 
with a mass cut-off of any finite size $K$.
We first recall an elementary inequality, 
which is a direct consequence of the mean value theorem and 
 the Young's inequality.
Given $p >2$ and $\eps >0$, there exists $C_\eps$ such that
\begin{align}
\label{Young}
|z_1 + z_2|^p \le (1+ \eps)|z_1|^p + C_\eps |z_2|^p  
\end{align}

\noi
holds uniformly in $z_1,z_2 \in \mathbb C$.
From \eqref{Rp}, \eqref{Young},
Proposition \ref{PROP:T},
and the fact
\[
\{\|Y_s(1) +I_s(\theta)(1)\|_{L^2(\T^d)}\le K\} \subset 
\{ \|I_s(\theta)(1)\|_{L^2(\T^d)}\le K+\| Y_s(1) \|_{L^2 (\T^d)} \},
\]
we obtain
\begin{align}
\label{var5}
& R_p \big(Y_s(1) +I_s(\theta)(1)\big) 
\cdot  \ind_{\{\|Y_s(1) +I_s(\theta)(1)\|_{L^2(\T^d)}\le K\}} \notag\\
& \le (1+ \eps) R_p \big( I_s(\theta)(1)\big) 
\cdot  \ind_{\{\|I_s(\theta)(1)\|_{L^2(\T^d)}\le K+\| Y_s(1) \|_{L^2 (\T^d)} \}}   + C_\eps R_p(Y_s(1)) \notag\\
& \le \frac{1+\eps}p (C_{\textup{GNS}} + \delta) (K+\| Y_s(1) \|_{L^2 (\T^d)})^{2+\frac{p-2}{2s}(2s-d)}  \| I_s(\theta)(1)\|^{\frac{(p-2)d}{2s}}_{\dot H^s (\T^d)}  \notag\\
& \hphantom{XXXXX} + C_\delta (K+ \| Y_s(1) \|_{L^2 (\T^d)})^p + C_\eps R_p(Y_s(1)). \notag\\
\intertext{Noting that $\frac{(p-2)d}{2s} < 2$ in this case, 
we apply Young's inequality to continue with}
& \le  C  + C \| Y_s(1)\|_{L^2 (\T^d)}^{2+\frac{4s(p-2)}{4s-(p-2)d}} 
+ \frac14 \| I_s(\theta)(1)\|_{\dot H^s (\T^d)}^2 
 +C  \| Y_s(1) \|_{L^2 (\T^d)}^p
+ C R_p(Y_s(1))
\end{align}

\noi
where $C$ is a constant depending on $\eps,\delta, p,d,s,\|Q\|_{L^2}$, and $K$.
By collecting \eqref{var4}, \eqref{var5} and Lemma \ref{LEM:bounds},
we arrive at
 \begin{align*}
\begin{split}
- & \log  \E_{\mu_s}  \bigg[ \exp \Big(R_p(u) \cdot  \ind_{\{\|u\|_{L^2(\T^d)}\le K\}} \Big)\bigg] \\
& \ge \inf_{\theta \in \mathbb{H}_a} \E \bigg[ - C - C \| Y_s(1)\|_{L^2 (\T^d)}^{2+\frac{4s(p-2)}{4s-(p-2)d}}  
 - C  \| Y_s(1) \|_{L^2 (\T^d)}^p  
 - C R_p(Y_s(1)) \\
& \hphantom{XXXXXXXX} 
- \frac14 \| I_s(\theta)(1)\|_{\dot H^s (\T^d)}^2 + \frac12 \int_0^1 
 \| \theta (t) \|_{L^2_x}^2 dt \bigg]\\
 & \ge \inf_{\theta \in \mathbb{H}_a} \E \bigg[ - C - C \| Y_s(1)\|_{L^2 (\T^d)}^{2+\frac{4s(p-2)}{4s-(p-2)d}}  
 - C  \| Y_s(1) \|_{L^2 (\T^d)}^p  
 - C \|Y_s(1)\|_{L^p(\T^d)}^p \\
& \hphantom{XXXXXXXXXXX} 
+ \frac14  \int_0^1 
 \| \theta (t) \|_{L^2_x}^2 dt \bigg]\\
 & \ge  \E \Big[  - C - C \|Y_s(1)\|_{L^p(\T^d)}^p - C  \| Y_s(1) \|_{L^2 (\T^d)}^p - C \| Y_s(1)\|_{L^2 (\T^d)}^{2+\frac{4s(p-2)}{4s-(p-2)d}}
   \Big] \\
 &\ge -  C - 2CC_{s,p} - C C_{s,2+\frac{4s(p-2)}{4s-(p-2)d}}  > -\infty,
\end{split}
\end{align*}

\noi
where $C_{s,r}$ is defined in Lemma \ref{LEM:bounds} (i). 
Thus we finish the proof of \eqref{var3} in the subcritical case.

\medskip

\noi
{\it \underline{Case 2}: critical interaction $p = \frac{4s}d +2$.}
We shall prove \eqref{var3} below the critical mass threshold 
$K < \|Q\|_{L^2 (\R^d)}$.
To get the sharp mass threshold,
we view $P_{\ge N} Y_s (1)$ as a perturbation instead.
It turns out that as $N$ is getting larger,
the probability of $P_{\ge N} Y_s (1)$ being large
shrinks exponentially to zero. See \eqref{Hsbound2}.

Since $s > \frac{d}2$, 
it follows that
\[
\lim_{N \to \infty} \| P_{\ge N} Y_s(1) \|_{L^2 (\T^d)} = 0,
\]

\noi
almost surely.
Therefore, given small $\eps >0$, 
for $\o \in \O$ almost sure, 
there exists an unique $N_\eps : = N_\eps(\o)$ such that
\begin{align}
\label{Neps}
\| P_{\ge \frac{N_\eps}2} Y_s(1) \|_{L^2 (\T^d)} > \eps \,\, \textup{ and } \,\,
\| P_{> N_\eps} Y_s(1) \|_{L^2 (\T^d)} \le \eps.
\end{align}


\noi
Similar argument as before with \eqref{Young}, Proposition \ref{PROP:T}, and \eqref{Neps},
yields that
\begin{align}
\label{var5a}
\begin{split}
& R_p \big(Y_s(1) +I_s(\theta)(1)\big) 
\cdot  \ind_{\{\|Y_s(1) +I_s(\theta)(1)\|_{L^2(\T^d)}\le K\}} \\
& \le (1+ \eps) R_p \big( P_{\le N_\eps} Y_s(1) + I_s(\theta)(1)\big) 
\cdot  \ind_{\{\|P_{\le N_\eps} Y_s(1) + I_s(\theta)(1)\|_{L^2(\T^d)}\le K+\eps \}}   + C_\eps R_p(Y_s(1))\\
& \le \frac{1+\eps}p (C_{\textup{GNS}} + \delta) (K+\eps)^{p-2} ( \| P_{\le N_\eps} Y_s(1)\|_{\dot H^s (\T^d)}  + \| I_s(\theta)(1)\|_{\dot H^s (\T^d)} )^2 \\
& \hphantom{XXXXX} + C_\eps R_p(Y_s(1))  + C_\delta (K+ \eps)^p\\
& \le \frac{(1+\eps)^2}p (C_{\textup{GNS}} + \delta) (K+\eps)^{p-2}  \| I_s(\theta)(1)\|^{2}_{\dot H^s (\T^d)} + C \| P_{\le N_\eps} Y_s(1)\|^{2}_{\dot H^s (\T^d)} \\
& \hphantom{XXXXX} + C_\eps R_p(Y_s(1))  + C_\delta (K+ \eps)^p,
\end{split}
\end{align}

\noi
where $C$ is a constant depending on $\eps,\delta, p,d,s,\|Q\|_{L^2}$, and $K$.
Since $C_{\textup{GNS}} = \frac{p}2 \|Q\|_{L^2}^{2-p}$,  $K < \| Q\|_{L^2 (\R^d)}$ and $p >2$, 
there exist $\eta , \eps, \delta >0$ such that
\begin{align}
\label{eta}
\frac{(1+\eps)^2}p (C_{\textup{GNS}} + \delta) (K+\eps)^{p-2}  < \frac{1-\eta}2.
\end{align}

\noi
By collecting \eqref{var4}, \eqref{var5a}, \eqref{eta},
and Lemma \ref{LEM:bounds},
we arrive at
 \begin{align*}
\begin{split} 
- & \log  \E_{\mu_s}  \bigg[ \exp \Big(R_p(u) \cdot  \ind_{\{\|u\|_{L^2(\T^d)}\le K\}} \Big)\bigg]\\
& \ge \inf_{\theta \in \mathbb{H}_a} \E \bigg[  - \frac{1-\eta}{2} \| I_s(\theta)(1)\|_{\dot H^s (\T^d)}^2 - C_\eps \| P_{\le N_\eps} Y_s(1)\|_{\dot H^s (\T^d)}^2 - C_\delta (K+\eps)^p  \\
& \hphantom{XXXXXXXXXX} -  C_\eps R_p(Y_s(1)) + \frac12 \int_0^1 
 \| \theta (t) \|_{L^2_x}^2 dt  \bigg]\\
 & \ge  \inf_{\theta \in \mathbb{H}_a}  \E \bigg[- C_\delta (K+\eps)^p -  C_\eps R_p(Y_s(1)) - C_\eps \| P_{\le N_\eps} Y_s(1)\|_{\dot H^s (\T^d)}^2  + \frac{\eta}{2} \int_0^1 
 \| \theta (t) \|_{L^2_x}^2 dt \bigg] \\
 & \ge  \E \Big[ - C_\delta (K+\eps)^p -  C_\eps R_p(Y_s(1)) - C_\eps \| P_{\le N_\eps} Y_s(1)\|_{\dot H^s (\T^d)}^2  \Big] \\
 & \ge  - C_\delta (K+\eps)^p -  C_\eps C_{s,p} - 
 C_\eps \E \big[  \| P_{\le N_\eps} Y_s(1)\|_{\dot H^s (\T^d)}^2  \big],
\end{split}
\end{align*}

\noi
where $C_{s,p}$ is given in \eqref{bounds1}.
We remark that $Y_s(1) \notin \dot H^s (\T^d)$ almost surely.
Therefore, to prove \eqref{var3},
it still needs to show that
\begin{align}
\label{Hsbound}
\E \big[ \| P_{\le N_\eps} Y_s(1)\|_{\dot H^s (\T^d)}^2 \big] 
 < \infty,
\end{align}

\noi
where $N_\eps$ is a random variable given by \eqref{Neps}.

Noting $Y_s(1)$ is a mean-zero random variable,
we may decompose $\O$ (by ignoring a zero-measure set) as
\begin{align}
\label{decom}
\O = \bigcup_{N \ge 1} \O_N,
\end{align}

\noi
where
\begin{align}
\label{ON}
\Omega_N = \Big\{ \o\in \O : N_\eps (\o) \in \big[\tfrac{N}2, N\big) \Big\}.
\end{align}

\noi
By \eqref{decom} and H\"older's  inequality, 
we have
\begin{align}
\label{Hsbound1}
\begin{split}
\E \big[ \| P_{\le N_\eps} Y_s(1)\|_{\dot H^s (\T^d)}^2 \big]  
& \le \sum_{N\ge 1} \E  \big[ \| P_{\le N} Y_s(1)\|_{\dot H^s (\T^d)}^2 \cdot \ind_{\O_N} \big] \\
& \le \sum_{N\ge 1} N^{2s} \E \big[ \| P_{\le N} Y_s(1)\|_{L^2 (\T^d)}^2 \cdot \ind_{\O_N} \big]\\
& \le \sum_{N\ge 1} N^{2s} \Big( \E \big[ \| Y_s(1)\|_{L^2 (\T^d)}^4 \big] \Big)^{\frac12} \cdot \P({\O_N})^{\frac12}
\\
& \le C_{s,4}^{\frac12} \sum_{N\ge 1} N^{2s}  \P({\O_N})^{\frac12},
\end{split}
\end{align}

\noi
where $C_{s,4}$ is given in \eqref{bounds1}.
By a direct computation, 
we have 
\begin{align}
\label{L2bound}
\E \big[ \| P_{\ge \frac{N}4} Y_s(1) \|_{L^{2} (\T^d)}^2  \big]  \sim N^{d-2s}.
\end{align}

\noi
It then follows from \eqref{Neps}, \eqref{ON},
H\"older's  inequality, Lemma \ref{LEM:Fer},
and \eqref{L2bound}, that
\begin{align}
\label{Hsbound2}
\begin{split}
\P(\O_N) & \le \P\big( \big\{ \| P_{\ge \frac{N}4} Y_s(1) \|_{L^2} > \eps \big\} \big) \\
& \les \exp \bigg\{-c \Big( \frac{\eps}{\E \big( \| P_{\ge \frac{N}4} Y_s(1) \|_{L^{2} (\T^d)}  \big) }\Big)^2 \bigg\} \\
& \les \exp \bigg\{- \Big( \frac{c \eps^2}{\E \big[ \| P_{\ge \frac{N}4} Y_s(1) \|_{L^{2} (\T^d)}^2  \big] } \bigg\} \\
& \les e^{-\tilde c \eps^2 N^{2s -d}},
\end{split}
\end{align}

\noi
where $c$ and $\tilde c$ are constant.
By collecting \eqref{Hsbound1} and \eqref{Hsbound2},
we conclude that
\begin{align*}
\begin{split}
\E \big[ \| P_{\le N_\eps} Y_s(1)\|_{\dot H^s (\T^d)}^2 \big]  
 \le C_{s,4}^{\frac12} \sum_{N\ge 1} N^{2s}  e^{- \frac{\tilde c}2 \eps^2 N^{2s -d}} < \infty,
\end{split}
\end{align*}

\noi
which finishes the proof of \eqref{Hsbound},
and thus \eqref{var3} in the critical case.

\medskip 

Therefore, we finish the proof of Theorem 1.1 -(i) and the first half of (ii).
\end{proof}
\medskip

\subsection{Non-integrability}

In this subsection,
we prove the rest of Theorem \ref{THM:main},
i.e. the non-integrability part of (ii) and (iii).
In particular, we show that the partition function 
\begin{align}
\label{part}
Z_{s,p,K}=\E_{\mu_s}
\Big[\exp({R_p(u)})
\ind_{\{\|u\|_{L^2(\T^d)}\le K\}}\Big] = \infty
\end{align}

\noi
under either of the following conditions
\begin{align}
\begin{split}
&\textup{(i) critical nonlinearity: } p = \frac{4s}d +2 \text{ and } K > \| Q\|_{L^2(\R^d)};\\
&\textup{(ii) super-critical nonlinearity: } p > \frac{4s}d +2 \text{ and any } K > 0.
\end{split}
\label{conditions}
\end{align}

\noi
Here $Q$ is the optimizer of the GNS 
inequality given in Theorem \ref{THM:fGNS} 
and Remark \ref{RMK:scaling}.
To prove \eqref{part},
we construct, within the ball $\{\|u\|_{L^2(\T^d)}\le K\}$, 
a sequence of 
drift terms given by perturbed scaled ``solitons", 
along which the variational formula \eqref{var} diverges.
The existence of such a sequence of scaled solitons
is guaranteed by the following lemma;

\begin{lemma}
\label{LEM:soliton}
Assume \eqref{conditions} holds.
Then, there exist 
a series of functions $\{W_\rho\}_{\rho >0} \subset H^s (\T^d) \cap L^p (\T^d)$ such that
\begin{align}
\label{soliton}
\begin{split}
\textup{ (i) } &H_{\T^d} (W_\rho)  \le - A_1 \rho^{-\frac{dp}2 + d},\\
\textup{ (ii) } &\| W_\rho\|_{L^p (\T^d)}^p \le A_2 \rho^{-\frac{dp}2 + d},\\
\textup{ (iii) } &\| W_\rho \|_{L^2(\T^d)} \le K - \eta,\\
\textup{ (iv) } & P_0 W_\rho \les 1 ,
\end{split}
\end{align} 

\noi
where $H_{\T^d}$ is the Hamiltonian functional given in \eqref{Hamiltonian},
and $A_1, A_2, A_3, \eta > 0$ are constant uniformly in 
sufficiently small $\rho > 0$.
\end{lemma}

In the next lemma, 
we construct an approximation $Z_M$ to $Y_s (1)$ in \eqref{Yt} through solving a stochastic differential equation.
These $Z_M$ act as controllable stochastic perturbations in defining the drift terms. See \eqref{theta} and \eqref{Itheta} in the following.
Similar approximation has appeared in \cite{OST}.
\begin{lemma}
\label{LEM:appro}
Given $s > \frac{d}2$ and a dyadic number $M \sim \rho^{-1} \gg 1$, 
define the $Z_M(t)$ by its Fourier coefficients:
Let $\ft Z_M (n,t)$ for $0 < |n| \le M$ be as follows:
\begin{align}
    \label{ODE}
    \begin{cases}
    d \ft Z_M (n,t) = |n|^{-s} M^{\frac{d}2} (\ft Y_s (n,t) - \ft Z_M (n,t)) dt\\
    \ft Z_M |_{t=0} =0,
    \end{cases}
\end{align}

\noi
and $\ft Z_M (n,t) = 0$ for $n=0$ and $|n| > M$.
Then the following holds:
\begin{align}
    \E & \big[\| Z_M  (1) - Y_s (1) \|_{L^p(\T^d)}^p \big] \les \max(M^{-s+ \frac{d}2}, M^{-\frac{d}2 +})^{\frac{p}2}, \text{ for } p\ge 1, \label{Lp}\\
    \E & \bigg[\Big\| D^s \frac{d}{dt} Z_M (t) \Big\|_{L^2(\T^d)}^2 \bigg] \les  \max (M^{\frac{3d}2 -s}, M^{\frac{d}2 +}), \label{dZL2}
\end{align}

\noi
for any $M \gg 1$.
\end{lemma}

The proofs of Lemma \ref{LEM:soliton}
and Lemma \ref{LEM:appro} will be postponed to the next subsection.
Now 
we are ready to prove the rest of Theorem \ref{THM:main}.

\begin{proof}[Proof of Theorem \ref{THM:main} - the second half of \textup{(ii) $and$ (iii)}]
We shall prove \eqref{part} under conditions \eqref{conditions}.
Observing that
\begin{align*}
\E_{\mu_s}  \Big[ \exp (R_p(u)) \cdot  \ind_{\{\|u\|_{L^2(\T^d)}\le K\}}\Big]
\ge \E_{\mu_s}  \Big[ \exp \Big(R_p(u) \cdot  \ind_{\{\|u\|_{L^2(\T^d)}\le K\}} \Big)\Big] - 1,
\end{align*}

\noi
then \eqref{part} follows from
\begin{align}
\label{var10}
\E_{\mu_s}  \bigg[ \exp \Big(R_p(u) \cdot  \ind_{\{\|u\|_{L^2(\T^d)}\le K\}} \Big)\bigg] = \infty.
\end{align}


To apply Lemma \ref{LEM:var},
we construct the series of drift terms as follows.
Let $W_\rho$ be as in Lemma \ref{LEM:soliton}, and 
\begin{align}
\label{theta}
\theta (t) \in \Big\{ -  D^s \frac{d}{dt} Z_M (t) +  D^s W_\rho \Big\}_{\rho >0},
\end{align}

\noi 
where $\rho \ll 1$ and $ M \sim \rho^{-1}$ is a dyadic number.
From \eqref{theta}, we have
\begin{align}
    \label{Itheta}
    \begin{split}
        I_s (\theta) (1) & =  \int_0^1  D^{-s} P_{\neq 0} \theta (t) dt \\
        & = \int_0^1 (P_{\neq 0} W_\rho - \frac{d}{dt} Z_M (t)) dt\\
        & = P_{\neq 0} W_\rho - Z_{M}(1).
    \end{split}
\end{align}
Thus, from Lemma \ref{LEM:var},
\eqref{theta}, 
and \eqref{Itheta} we have
\begin{align}
\label{var13}
\begin{split}
- & \log  \E_{\mu_s}  \bigg[ \exp \Big(R_p(u) \cdot  \ind_{\{\|u\|_{L^2(\T^d)}\le K\}} \Big)\bigg]\\
& = \inf_{\theta \in \mathbb H_a} \E \bigg[ \Big( - R_p (Y_s(1) + I_s (\theta) (1))  \cdot \ind_{\{\|Y_s(1) + I_s (\theta) (1)\|_{L^2 (\T^d)}\le K\}}
+ \frac12 \int_0^1 \| \theta (t)\|_{L^2 (\T^d)}^2\Big)    \bigg] \\
 &   \le \inf_{0 < \rho \ll 1} \E \bigg[ \Big( - R_p( Y_s(1) - Z_M (1) + P_{\neq 0} W_\rho )  \cdot \ind_{\{\| Y_s(1) - Z_M (1) + P_{\neq 0} W_\rho \|_{L^2 (\T^d)}\le K\}}\\
 & \hphantom{XXXXXX}
+ \frac12 \int_0^1 \Big\| -  \frac{d}{dt} Z_M (t) +    P_{\neq 0} W_\rho \Big\|_{\dot H^s  (\T^d)}^2\Big) dt   \bigg]\\
 & = \inf_{0 < \rho \ll 1} \E \bigg[ 
 \Big(- R_p (W_\rho) + \frac12 \| W_\rho\|_{\dot H^s  (\T^d)}^2 \Big)\\
&\hphantom{XX} + \Big( R_p (W_\rho) - R_p (P_{\neq 0} W_\rho)\Big) \\
& \hphantom{XX}+  \Big( R_p (P_{\neq 0}  W_\rho) - R_p ( Y_s(1) - Z_M (1) + P_{\neq 0} W_\rho ) \Big)  \cdot \ind_{\{\| Y_s(1) - Z_M (1) + P_{\neq 0} W_\rho \|_{L^2 (\T^d)}\le K\}}\\
& \hphantom{XX} + R_p(P_{\neq 0}  W_\rho) \cdot \ind_{\{\| Y_s(1) - Z_M (1) + P_{\neq 0}  W_\rho \|_{L^2 (\T^d)} > K\}} \\
 & \hphantom{XX}
+ \frac12 \int_0^1 \Big\| -  \frac{d}{dt} Z_M (t) \Big\|_{\dot H^s  (\T^d)}^2 - 2 \Big\langle   \frac{d}{dt} Z_M (t) ,   W_\rho \Big\rangle_{\dot H^s (\T^d)} \Big) dt   \bigg] \\
& = \inf_{0 < \rho \ll 1} ( \textup{A + B  + C + D + E}).
\end{split}
\end{align}

\noi
In what follows, 
we consider these terms one by one for $0 < \rho \ll 1$.

For term (A), from \eqref{soliton} - (i),
we have
\begin{align}
    \label{term1}
    \textup{A}  = - R_p (W_\rho) + \frac12 \| W_\rho\|_{\dot H^s (\T^d)}^2 = H_{\T^d} (W_\rho) \les - \rho^{-\frac{dp}2 + d}.
\end{align}

\noi
For term (B),
from \eqref{soliton} - (iv) and the mean value theorem,
we have
\[
\begin{split}
\textup{B} & = \frac1{p} \int_{\T^d} \big(  |W_\rho|^p - |W_\rho - P_0 W_\rho|^p \big) dx \\
& \les 
\int_{\T^d} \big( |P_0 W_\rho |^p + |P_0 W_\rho| | W_\rho|^{p-1} \big) dx \\
& \les 1 + \|W_\rho\|_{L^{p-1} (\T^d)}^{p-1},
\end{split}
\]

\noi
Then, by interpolating \eqref{soliton} - (ii) and (iii),
we obtain
\begin{align}
    \label{term2}
    \textup{B} \les   \rho^{-\frac{d(p-1)}2 + d}.
\end{align}

\noi
For term (C),
by using the mean value theorem we see that
\[
\begin{split}
\int_{\T^d} \big( & |P_{\neq 0} W_\rho|^p - |Y_s(1) - Z_M (1) + P_{\neq 0}  W_\rho|^p \big) dx \\
& \les 
\int_{\T^d} \big( |Y_s(1) - Z_M (1) |^p + |Y_s(1) - Z_M (1)| | P_{\neq 0} W_\rho|^{p-1} \big) dx ,
\end{split}
\]

\noi
which together with Lemma \ref{LEM:appro} and Lemma \ref{LEM:soliton} gives
\begin{align}
\label{term3}
    \begin{split}
    \textup{C} = \E & \bigg[ \Big( R_p (P_{\neq 0} W_\rho) - R_p ( Y_s(1) - Z_M (1) + P_{\neq 0} W_\rho ) \Big)  \cdot \ind_{\{\| Y_s(1) - Z_M (1) + W_\rho \|_{L^2 (\T^d)}\le K\}}\bigg] \\
        & \les \int_{\T^d} \big( \E \big[ |Y_s(1) - Z_M (1) |^p \big] + \E \big[ |Y_s(1) - Z_M (1)|\big] | P_{\neq 0} W_\rho|^{p-1} \big) dx \\
        & \les \max(M^{-s+\frac{d}2}, M^{-\frac{d}2+})^{\frac{p}2}
        + \max(M^{-s+\frac{d}2}, M^{-\frac{d}2+})^{\frac{1}2} \| P_{\neq 0}  W_\rho \|_{L^{p-1}  (\T^d)}^{p-1}  \\
        & \les 
       ( \| P_{\neq 0}  W_\rho \|_{L^{p-1}  (\T^d)}^{p-1}   - \| W_\rho\|_{L^{p-1} (\T^d)}^{p-1} ) + \| W_\rho\|_{L^{p-1} (\T^d)}^{p-1}  \\
       & \les \rho^{-\frac{d(p-1)}2 + d},
    \end{split}
\end{align}

\noi
where in the last step, to bound $(\| P_{\neq 0}  W_\rho \|_{L^{p-1}  (\T^d)}^{p-1}   - \| W_\rho\|_{L^{p-1} (\T^d)}^{p-1} )$, 
we used a similar argument as in estimating term (B).
Now we turn to term (D),
by using Chebyshev's inequality, \eqref{soliton} - (iii), \eqref{Lp}, and \eqref{term2},
we have
\begin{align}
    \begin{split}
    \textup{D} &= \E \big[ R_p(P_{\neq 0} W_\rho) \cdot \ind_{\{\| Y_s(1) - Z_M (1) + P_{\neq 0}W_\rho \|_{L^2 (\T^d)} > K\}} \big] \\
    & \le  R_p(P_{\neq 0}  W_\rho) \cdot \E \big[ \ind_{\{\| Y_s(1) - Z_M (1)\|_{L^2 (\T^d)}  > K - \| W_\rho \|_{L^2 (\T^d)}\}} \big]\\
    & \le R_p(P_{\neq 0}  W_\rho) \frac{\E [\| Y_s(1) - Z_M (1)\|_{L^2 (\T^d)}^2] }{(K - \| W_\rho \|_{L^2 (\T^d)})^2}\\
    & \les  \rho^{-\frac{dp}2 + d} \max(M^{-s+\frac{d}2}, M^{-\frac{d}2+}) \\
    & \les \max ( \rho^{-\frac{d(p-1)}2 + s}, \rho^{-\frac{d(p-3)}2 +}),
    \end{split}
    \label{term4}
\end{align}

\noi
where in the last step we use the relation $M \sim \rho^{-1}$.
For term (E),
from \eqref{ODE} and \eqref{dZL2},
we have
\begin{align}
    \label{term5}
    \begin{split}
    \textup{E} 
    & = \frac12 \int_0^t \E \bigg[  \Big\| -  \frac{d}{dt} Z_M (t) \Big\|_{\dot H^s  (\T^d)}^2 \bigg] dt  \les \max( \rho^{-\frac{3d}2 + s}, \rho^{-\frac{d}2 +}),
    \end{split}
\end{align}

\noi 
where we used the fact that $Z_M$, removing the zero frequency, is a mean zero Gaussian random variable.
By collecting estimates  \eqref{term1}, \eqref{term2}, \eqref{term3}, \eqref{term4}, and \eqref{term5},
we conclude that 
\begin{align}
    \label{abcd}
    \textup{A + B + C + D + E} \les - \rho^{-\frac{dp}2 + d},
\end{align}

\noi
where we used \eqref{conditions} and the assumption $s > \frac{d}2$.

Finally,
the desired estimate \eqref{var10} follows from \eqref{var13} and \eqref{abcd}.
We thus finish the proof of Theorem \ref{THM:main}.
\end{proof}

\subsection{Proof of the auxiliary lemmas}
It remains to prove Lemmas \ref{LEM:soliton} and \ref{LEM:appro},
which is the main purpose of this subsection.
We first present the proof of Lemma \ref{LEM:soliton}.

\begin{proof}[Proof of Lemma \ref{LEM:soliton}]
Define $W_{\rho} \in H^s (\T^d)$ by
\begin{align}
\label{W}
W_{\rho} (x) : = \al \rho^{-\frac{d}2} \phi_\dl (x) Q(\rho^{-1}x) ,
\end{align}

\noi
where $\phi_\dl$ is the same as in Proposition \ref{PROP:T}, $\al >0$ is to be determined later and $Q$ is given in Theorem \ref{THM:fGNS} and 
Remark \ref{RMK:scaling}.
Then (iv) follows directly from $\|W_\rho\|_{L^1(\T^d)} \les \|W_\rho\|_{L^2(\T^d)}$.
We only consider (i) -- (iii) in what follows.
We distinguish two cases based on the conditions
in \eqref{conditions}:

\noi
{\it \underline{Case 1}: critical nonlinearity}. 
In this case, we have $p = \frac{4s}d +2 > 2 \text{ and } 
K > \| Q\|_{L^2(\R^d)}$.
Fix $\al >1$ such that 
\begin{align}
\label{eta1}
\| \al Q\|_{L^2 (\R^d)} = \al \| Q\|_{L^2 (\R^d)}  = K - \eta,
\end{align}

\noi
where $\eta$ is given in \eqref{soliton}.
Recall that $H_{\R^d}(Q) = 0$ from Remark \ref{RMK:scaling}.
We then have
\begin{align*}
H_{\R^d}(\al Q) = \frac{\al^2}2 \int_{\R^d} |D^{s} Q|^2 d x - \frac{\al^p}p \int_{\R^d} |Q|^p dx < 0. 
\end{align*}

\noi
Then, it follows from Remark \ref{RMK:compare} that
\begin{align*}
 H_{\T^d} (W_{\rho})&=\frac{\al^2}2\int_{\mathbb{T}^d}|D^s ( \phi_\dl Q_{\rho})|^2dx
 -\frac{\al^p}p\int_{\mathbb{T}^d}|\phi_\dl (x) Q_{\rho} (x)|^pdx\\
& \le \frac{\al^2}{2}  \|D^s (\phi_\dl Q_\rho) \|^2_{L^2(\R^d)}
-\frac{\al^p}{p} \int_{ \mathbb{T}^d}|\phi_\dl (x) Q_\rho (x)|^pdx,\\
\intertext{where $Q_\rho = \rho^{-\frac{d}2}  Q(\rho^{-1}x)$. By the fractional Leibnize rule \cite{KPV,Li}
and Sobolev embedding, we may continue with}
& \le \frac{\al^2 + \eps}{2}  \|D^s Q_\rho \|^2_{L^2(\R^d)}
-\frac{\al^p}{p} \int_{ \mathbb{T}^d}|\phi_\dl (x) Q_\rho (x)|^pdx + C \| Q_\rho \|_{W^{0+,2}}^2,\\
\intertext{then by interpolation we can continue with}
& \le \frac{\al^2 + 2\eps}{2}\rho^{-2s} \|D^s Q\|^2_{L^2(\R^d)}
-\frac{\al^p}{p}\rho^{-\frac{dp}2+d} \int_{ \mathbb{R}^d}| \phi_\dl (\rho x) Q (x)|^pdx + C \| Q_\rho \|_{L^2 (\R^d)}^2,\\
\intertext{we have $\int_{ \mathbb{R}^d}| \phi_\dl (\rho x) Q (x)|^pdx > \frac1{\al^\eps} \int_{ \mathbb{\R}^d}|Q (x)|^pdx$, provided that $\rho$ is sufficiently small. Thus, combining with \eqref{scaling} and the fact $2s = \frac{dp}2 -d$
from \eqref{conditions} - (i), we may continue with}
&\le \left(\frac{\al^2 + 2\eps }{p}-\frac{\al^{p-\eps}}{p}\right) \rho^{-\frac{dp}2+d}\|Q\|_{L^p(\R^d)}^p + C \|Q\|_{L^2(\R^d)}^2,
\end{align*}

\noi
which finishes the proof of \eqref{soliton} - (i) by choosing $\eps$ small
enough and setting
\[ A_1 := \left(\frac{\al^{p-\eps}}{p} - \frac{\al^2 + 2\eps}{p} - \eps \right) 
\|Q\|_{L^p(\R^d)}^p. \]

\noi
As to \eqref{soliton} - (ii) and (iii), we note that 
\begin{equation*}
\begin{split}
 \left\|W_{\rho}\right\|_{L^p(\mathbb{T}^d)}^p & \leq
\al^p \rho^{-\frac{pd}2+d}\|Q\|_{L^p(\R^d)}^p =  A_2 \rho^{-\theta},\\
 \left\|W_{\rho}\right\|_{L^2(\mathbb{T}^d)} & \le 
 \al \|Q_\rho\|_{L^2(\R^d)}
 = \al \|Q\|_{L^2(\R^d)} = K -\eta,
\end{split}
\end{equation*}

\noi
with $ A_2 : = \al^p \|Q\|_{L^p(\R^d)}^p $ and $\eta$ 
being the one in \eqref{eta1}.
%
Thus, we finish the proof.

\medskip

\noi
{\it \underline{Case 2}: super-critical nonlinearity}.
In what follows, we assume $p > \frac{4s}d +2 $.
It only needs to prove \eqref{soliton} - (i) and (iii),
since (ii) follows the same way as that of Case 1.
Given $K>0$,
we choose $\al \ll1$ in \eqref{W} so that
\begin{align*}
\left\|W_{\rho}\right\|_{L^2(\mathbb{T}^d)} \le \|W_\rho\|_{L^2 (\R^d)} = \al \|Q\|_{L^2 (\R^d)} < K - \eta,
\end{align*}

\noi
which gives \eqref{soliton} - (iii).
Similar computation as in the previous case, we have
\begin{align*}
 H_{\T^d} (W_{\rho}) &=\frac{\al^2}2\int_{\mathbb{T}^d}
 |D^s (\phi_\dl Q_{\rho})|^2dx
 -\frac{\al^p}p\int_{\mathbb{T}^d}|\phi_\dl Q_{\rho}|^pdx\notag\\
& \le  \frac{\al^2 + 2\eps}{2}\rho^{-2s} \|D^s Q\|^2_{L^2(\R^d)}
-\frac{\al^{p-\eps}}{p}\rho^{-\frac{dp}2+d} \| Q\|^p_{L^p(\R^d)} + C \| Q_\rho \|_{L^2 (\R^d)}^2,\notag\\
\intertext{for sufficiently small $\rho$ and $\dl$. 
Also, 
note that $p>\frac{4s}d+2$ implies $-2s > -\frac{dp}2+d$.
Thus, recalling \eqref{scaling}, we may continue with}
&\le \left(\frac{\al^2 + 2\eps}{p}\rho^{-2s}-\frac{\al^{p-\eps}}{2p}\rho^{-\frac{dp}2+d}\right)\|Q\|_{L^p(\R^d)}^p + C \|Q\|_{L^2(\R^d)}^2 \notag\\
&\leq - \wt A_1\rho^{-\frac{dp}2+d},
\end{align*}

\noi
for sufficiently small $\rho>0$ and some constant $\wt A_1 >0$.
Thus, we obtain \eqref{soliton} - (i).
We finish the proof of Lemma \ref{LEM:soliton}.
\end{proof}

Next, we present the proof of Lemma \ref{LEM:appro}.

\begin{proof}
[Proof of Lemma \ref{LEM:appro}]
Let 
\begin{align}
    \label{Xn}
    X_n (t) = \ft Y_s (n,t) - \ft Z_M (n,t), \quad 0< |n| \le M.
\end{align}

\noi
Then, from \eqref{Yt} and \eqref{ODE},
we see that $X_n(t)$ solves
\[
\begin{cases}
d X_n (t) = - |n|^{-s} M^{\frac{d}2} X_n (t) dt + |n|^{-s} dB_n (t)\\
X_n(0) = 0
\end{cases}
\]

\noi
for $0 < |n| \le M$.
Solving the above stochastic differential equation yields
\begin{align}
    \label{Xn1}
    X_n (t) = |n|^{-s} \int_0^t e^{-|n|^{-s} M^{\frac{d}2} (t-t')} dB_n (t').
\end{align}

\noi 
Then, 
from \eqref{Xn} and \eqref{Xn1},
we have
\begin{align}
    \label{ZM}
    \ft Z_M (t) = \ft Y_s (n,t) - |n|^{-s} \int_0^t e^{-|n|^{-s} M^{\frac{d}2} (t-t')} dB_n (t'),
\end{align}

\noi
for $0 < |n| \le M$.
In what follows,
we show that $Z_M$ approximates to $Y_s$ as $M \sim \rho^{-1}$ tends to infinity.
From \eqref{ZM}, the independence of $\{B_n\}_{n \in \Z^d}$, 
and Ito's isometry,
we have
\begin{align}
    \begin{split}
    \E \big[ |Z_M (1) - Y_s (1)|^2 \big] 
    & =  \sum_{0< |n| \le M} |n|^{-2s} \int_0^t e^{-2|n|^{-s} M^{\frac{d}2} (t-t')} dt' + \sum_{|n| > M} |n|^{-2s}\\
    & \les  \sum_{0< |n| \le M} |n|^{-s} M^{-\frac{d}2} + M^{-2s+d}\\
    & \les  \max(M^{-s+ \frac{d}2}, M^{-\frac{d}2 +}),
    \end{split}
    \label{pL2}
\end{align}

\noi 
which is sufficient for \eqref{Lp} with $p=2$.

When $p=1$,
\eqref{Lp} follows from \eqref{pL2} together with H\"older's  inequality
\[
\E \big[ |Z_M (1) - Y_s (1)| \big] \les \Big( \E \big[ |Z_M (1) - Y_s (1)|^2 \big] \Big)^{\frac12}.
\]

\noi
Then the case for $1<p <2$ follows from interpolation.
When $p > 2$, we note that
$Z_M (1) - Y_s(1) \in \mathcal H_1$, 
homogeneous Wiener chaoses of order 1.
Then, 
by using Wiener chaos estimate \cite[Lemma I.22]{Simon},
we obtain
\[
\begin{split}
\E \big[ |Z_M (1) - Y_s (1)|^p \big] 
\les  \E \big[ |Z_M (1) - Y_s (1)|^p \big] \les  \Big( \E \big[ |Z_M (1) - Y_s (1) |^2 \big] \Big)^{\frac{p}2},
\end{split}
\]

\noi
which together with \eqref{pL2} implies \eqref{Lp} for $p > 2$.

Finally, we turn to \eqref{dZL2}.
From \eqref{ODE} and \eqref{Xn}, 
we have
\begin{align*}
    \begin{split}
    \E  \bigg[\Big\| D^s \frac{d}{dt} Z_M (t) \Big\|_{L^2(\T^d)}^2 \bigg] & = M^d \sum_{0 < |n| \le M} \E \big[ |X_n(t) |^2 \big] \\
    & = M^d \sum_{0 < |n| \le M} |n|^{-2s} \int_0^t e^{-2|n|^{-s} M^{\frac{d}2} (t-t')} dt'\\
    & \les M^d \sum_{0 < |n| \le M} |n|^{-s} M^{-\frac{d}2} \\
    & \les \max (M^{\frac{3d}2 -s}, M^{\frac{d}2 +}).
    \end{split}
\end{align*}

\noi
We finish the proof of \eqref{dZL2} and thus we conclude this lemma.
\end{proof}

\begin{ackno}\rm
The authors~would like to thank Tadahiro Oh for
his kind help during the preparation 
of the paper.
Y.W.~was supported by 
 the EPSRC New Investigator Award 
 (grant no.~EP/V003178/1).
The authors also would like to thank anonymous referees' comments which help to improve the presentation of this paper.
\end{ackno}

\end{document}